\documentclass{amsart}

\usepackage{pinlabel}

\usepackage{mathrsfs}
\usepackage{amsmath}
\usepackage{amsfonts}
\usepackage{amssymb}
\usepackage{amsthm}
\usepackage[all]{xy}
\usepackage{enumerate}
\usepackage{setspace}
\usepackage{stmaryrd}
\usepackage{tikz-cd}
\usepackage[pagewise]{lineno}

\newcommand{\eqn}[1]{\begin{displaymath}\begin{split}#1\end{split}\end{displaymath}}

\newtheorem{theorem}{Theorem}
\newtheorem{corollary}{Corollary}
\newtheorem{proposition}{Proposition}
\newtheorem{lemma}{Lemma}

\theoremstyle{definition}

\theoremstyle{definition}
\newtheorem{definition}{Definition}

\theoremstyle{definition}

\theoremstyle{definition}
\newtheorem{example}{Example}

\theoremstyle{definition}
\newtheorem{remark}{Remark}

\theoremstyle{definition}

\newcommand{\defn}[1]{\emph{#1}}

\newcommand{\N}{\mathbb{N}}
\newcommand{\C}{\mathbb{C}}
\newcommand{\R}{\mathbb{R}}

\newcommand{\Q}{\mathbb{Q}}
\newcommand{\Nminus}{\mathbb{N}\setminus\{0\}}

\makeatletter
\newcommand{\dotminus}{\mathbin{\mathrm{\@dotminus}}}

\newcommand{\@dotminus}{%
  \ooalign{\hidewidth\raise1ex\hbox{.}\hidewidth\cr$\m@th-$\cr}%
}
\makeatother
\newcommand{\dotsub}{\dotminus}
\newcommand{\pred}{\mathcal{P}}
\newcommand{\fun}{\mathcal{F}}
\newcommand{\con}{\mathcal{C}}
\renewcommand{\d}{\underline{d}}
\newcommand{\var}{\mathcal{V}}
\newcommand{\half}{\frac{1}{2}}
\newcommand{\bic}{\leftrightarrow}

\newcommand{\dyad}{\ensuremath{\mathrm{Dyad}}}
\newcommand{\p}{\underline{p}}
\newcommand{\q}{\underline{q}}
\renewcommand{\u}{\underline{u}}

\newcommand{\M}{\mathfrak{M}}

\newcommand{\Nstruct}{\mathfrak{N}}

\newcommand{\X}{\mathfrak{X}}

\newcommand{\ran}{\mathrm{ran}}

\newcommand{\id}{\mathrm{id}}

\newcommand{\href}{\text{ }}

\title[Generalized effective completeness for continuous logic]{Generalized effective completeness for continuous logic}
\author[Caleb Camrud]{Caleb Camrud}
\address{Department of Mathematics\\
Iowa State University\\\newline
Carver Hall\\ 411 Morrill Rd.\\ Ames, IA 50014\\ USA}
\email{ccamrud@iastate.edu}
\urladdr{https://cmhcamrud.org}

\begin{document}
\begin{abstract}
In this paper, we present a generalized effective completeness theorem for continuous logic. The primary result is that any continuous theory is satisfied in a structure which admits a presentation of the same Turing degree. It then follows that any decidable theory is satisfied by a computably presentable structure. This modifies and extends previous partial effective completeness theorems for continuous logic given by Calvert and Didehvar, Ghasemloo, and Pourmahdian.
\end{abstract}


\maketitle

\section{Introduction}

Completeness results relate theories to structures. \emph{Effective} completeness results relate \emph{decidable} theories to \emph{computable} structures. The first such result was given by Millar in \cite{Millar.1978}. But the method provided in that manuscript only applies to classical logic and classically computable structures, and hence cannot be directly applied to continuous logic and uncountable structures.

In \cite{BenYaacov.Berenstein.Henson.Usvyatsov.2008}, Ben Yaacov \textit{et al.} developed a model theory for metric structures using continuous first-order logic, and a completeness result was proven in \cite{BenYaacov.Pedersen.2010}. Calvert then extended this result to an effective version of completeness, relating decidable theories in continuous logic to probabilistically decidable structures.

\begin{theorem}[Theorem 4.5, \cite{Calvert.2011}]
    Let $T$ be a complete, decidable, continuous first-order theory. Then there is a probablistically decidable, continuous weak structure $\M$ such that $\M\vDash T$.
\end{theorem}

In the last decade, however, computable presentations rather than probabilistic decidability have become standard for the study of effectivity on metric structures (see, \emph{e.g.}, \cite{Brown.McNicholl.Melnikov.2020} and \cite{Franklin.McNicholl.2020}). We therefore sought to answer the following question: ``Is there an effective completeness theorem for continuous logic and \emph{computable presentations}?''

Didehvar, Ghasemloo, and Pourmahdian provided a partial answer to this question in \cite{Didehvar.Ghasemloo.Pourmahdian.2010}, implicitly with respect to computable presentations. The result proven was a qualified effective completeness result for the first-order rational Pavelka logic ($\mathbf{RPL} \forall$).

\begin{theorem}[Theorem 3.5, \cite{Didehvar.Ghasemloo.Pourmahdian.2010}]
    Every consistent, linear-complete, computably axiomatizable Henkin theory in $\mathbf{RPL} \forall$ has a decidable model.
\end{theorem}

Notably, the continuous logic of \cite{BenYaacov.Berenstein.Henson.Usvyatsov.2008} is a fragment of $\mathbf{RPL} \forall$, so the above theorem applies to continuous logic, as well. 

Our primary result can be considered as a generalization and strengthening of the above theorem: there is an effective procedure which, given a name of a continuous theory, produces a presentation of a metric structure which models that theory. It follows that the presentation is Turing reducible to the given name. Hence if the theory is decidable, the presentation is computable.

The result has four major upshots in comparison to the previous effective completeness results.
\begin{enumerate}[(i)]
    \item The given (name of a) continuous theory does not need to be either \emph{complete} or \emph{linear-complete}. Through the process described in Lemma \ref{lem:maxconsset}, any continuous theory may be effectively extended to a complete theory (though, as a \emph{caveat}, this extension is not unique).
    \item The given (name of a) continuous theory need not be \emph{decidable}. The generalized effective completeness theorem applies to any name of a continuous theory, and produces a presentation which is Turing reducible to that name.
    \item The generalized effective completeness theorem relates continuous theories to \emph{presentations} of metric structures, which have become the default for the study of effectivity on metric structures, rather than \emph{probabilistic decidability}.
    \item The generalized effective completeness theorem relates continuous theories to presentations of genuine metric \emph{structures}, rather than \emph{weak structures}, so the assumption of a countable universe may be dropped.
\end{enumerate}

This paper is organized as follows. In Section \ref{sec:back}, we introduce continuous logic, metric structures, and computable presentations. Section \ref{sec:completeres} recalls previous results in the model theory of metric structures. These results are then extended to important preliminary model-theoretic propositions in Section \ref{sub:mod}. Section \ref{sub:dec} follows to include our primary lemma, allowing us to uniformly effectively extend theories to complete theories. We then prove our main theorem, a generalized version of effective completeness, in Section \ref{sub:eff}. Standard effective completeness then follows from this as a corollary.
    

\section{Background}\label{sec:back}

\subsection{Continuous logic}

The \defn{logical symbols} of continuous logic consist of the following.
    \begin{itemize}
        \item $($ and $)$ are the \defn{parentheses}.
        \item $x,y,z,...$ are the \defn{variable symbols} ($\var$).
        \item $\neg$, $\half$, and $\dotsub$ are the \defn{connectives}.
        \item $\sup$ and $\inf$ are the \defn{quantifiers}.
    \end{itemize}

\begin{remark}\label{rem:morecon}
    In some versions of continuous logic, the set of connectives contains a distinguished symbol $\u$ for each continuous map $u:[0,1]^{\eta(u)}\to[0,1]$. The resulting set of well-formed formulas for such a logic is, however, uncountable and thus fails to perform effectively. Our choice of $\neg$, $\half$, and $\dotsub$ as the connectives was made for four reasons.
    \begin{enumerate}
        \item \label{it:int} $\neg$ plays precisely the role of classical negation ($\neg$) and $\dotsub$ of reverse implication ($\leftarrow$). The interpretation of the $\half$ operator is similarly intuitive, as will be shown in the following subsection.
        \item \label{it:full} In \cite{BenYaacov.Usvyatsov.2010}, it was shown that after interpretation, $\neg$, $\half$, and $\dotsub$ are dense in the set of all continuous maps on $[0,1]$. Thus finitary well-formed formulas in these connectives can approximate those in the wider set of connectives arbitrarily well. Such an approximation is, moreover, sufficient for completeness (as seen in \cite{BenYaacov.Pedersen.2010}).
        \item \label{it:eff} When a signature is effectively numbered, the sentences and well-formed formulas of that signature may be effectively enumerated.
        \item Because of \ref{it:int}, \ref{it:full}, and \ref{it:eff}, $\neg$, $\half$, and $\dotsub$ have become a somewhat canonical set of connectives for continuous logic.
    \end{enumerate}
\end{remark}

A \defn{signature} is a quintuple $L=\big(\pred,\fun,\con,\Delta,\eta\big)$
    such that each of the following hold.
    \begin{itemize}
        \item $\pred$, $\fun$, and $\con$ are mutually disjoint, and contain no logical symbols.
        \item $\Delta:\pred\cup\fun\to \N^\N$.
        \item $\eta:\pred\cup\fun\to\Nminus$.
        \item There is some $\d\in \pred$ such that $\Delta(\d)=\id_\N$ and $\eta(\d)=2$.
    \end{itemize}
    $\pred$ is the set of \defn{predicate symbols}, $\fun$ the set of \defn{function symbols}, and $\con$ the set of \defn{constant symbols}. $\Delta$ is the \defn{modulus map} and $\eta$ the \defn{arity map}. Each predicate (or function) symbol $F$ is an \emph{$\eta(F)$-ary} predicate (or function) symbol. 

For the remainder of this chapter, unless stated otherwise, we will assume we have a fixed signature $L$. The construction of \defn{terms} and \defn{well-formed formulas} (wffs) is straightforward. As are the definitions of \defn{free variables} and \defn{sentences}. Explicit definitions can be found in \cite{BenYaacov.Berenstein.Henson.Usvyatsov.2008}. A \defn{theory} is a set of sentences. Note, now, that the following syntax maps will be used as shorthand.
    \begin{center}
    \begin{tabular}{ |c|c| } 
     \hline
     \textbf{Shorthand} & \textbf{String} \\ 
     \hline
     $\varphi\vee\psi$ & $\neg\big((\neg\varphi)\dotsub \psi\big)$\\
     $\varphi\wedge\psi$ & $\varphi\dotsub(\varphi\dotsub \psi)$\\
     $\varphi\bic\psi$ & $(\varphi\dotsub\psi)\vee(\psi\dotsub\varphi)$\\
     $\vec{x}$ & $(x_0,...,x_{n})$ \\
     $\sup_{x_0,...,x_n}\varphi$ & $\sup_{x_0}\dots\sup_{x_n}\varphi$ \\
     $\inf_{x_0,...,x_n}\varphi$ & $\inf_{x_0}\dots\inf_{x_n}\varphi$ \\
     $\underline{0}$ & $\sup_x\d(x,x)$\\
     $\underline{1}$ & $\neg \ \underline{0}$\\
     $\varphi\dotplus\psi$ & $\neg\big((\underline{1}\dotsub \varphi)\dotsub \psi)\big)$\\
     $m\varphi$ & $\underbrace{\big(...(\varphi\dotplus \varphi)\dotplus\dots\dotplus \varphi\big)}_\text{$m$-many}$\\
     $\underline{2^{-k}}$ & $\underbrace{\half \dots \half}_\text{ $k$-many} \  \underline{1}$\\
     $\underline{\frac{\ell}{2^k}}$ & $\underbrace{\big(...(\underline{2^{-k}}\dotplus \underline{2^{-k}})\dotplus\dots\dotplus \underline{2^{-k}}\big)}_\text{$\ell$-many}$\\
     \hline
    \end{tabular}
    \end{center}

Some of the above syntax may seem loaded; this is for good reason, as will be noted in \ref{sec:metstruct}.

We present a simplified list of \defn{axiom schemata} for continuous logic, which are more parsimonious for effective constructions. In each of the following, $\varphi$, $\psi$, and $\theta$ range over arbitrary wffs. The first four schemata correspond to the classical propositional axioms.
\begin{enumerate}[I.]
    \item \label{ax:one} $(\varphi\dotsub \psi)\dotsub \varphi.$
    \item $\big((\theta\dotsub \varphi)\dotsub (\theta\dotsub \psi)\big)\dotsub (\psi\dotsub \varphi).$
    \item $\big(\varphi\dotsub (\varphi\dotsub \psi)\big)\dotsub\big(\psi\dotsub(\psi\dotsub\varphi)\big)$
    \item $(\varphi\dotsub \psi)\dotsub (\neg\varphi\dotsub \neg\varphi).$
\end{enumerate}

\noindent The next four correspond to the classical first-order axiom schemata. For every $x\in\var$ and term $t$,

\begin{enumerate}[I.]\addtocounter{enumi}{4}
    \item $(\sup_x\psi\dotsub\sup_x\varphi)\dotsub\sup_x(\psi\dotsub\varphi).$
    \item $\varphi[t/x]\dotsub \sup_x\varphi$, when this substitution is correct.
    \item $\sup_x\varphi\dotsub \varphi$, when $x$ is not free in $\varphi$.
    \item $\inf_x\varphi \bic \neg(\sup_x\neg\varphi)$.
\end{enumerate}

\noindent These two define the $\frac{1}{2}$ connective.

\begin{enumerate}[I.]\addtocounter{enumi}{8}
    \item $\frac{1}{2}\varphi\dotsub (\varphi\dotsub \frac{1}{2}\varphi).$
    \item \label{ax:ten} $(\varphi\dotsub \frac{1}{2}\varphi)\dotsub\frac{1}{2}\varphi.$
\end{enumerate}

\noindent The following define the predicate symbol $\d$. For every $x,y,z\in\var$,

\begin{enumerate}[I.]\addtocounter{enumi}{10}
    \item \label{ax:ref} $\d(x,x).$
    \item \label{ax:sym} $\d(x,y)\dotsub \d(y,x).$
    \item \label{ax:tri} $\big(\d(x,z)\dotsub\d(x,y)\big)\dotsub \d(y,z).$
\end{enumerate}

\noindent The next schema defines interaction between $\d$ and function symbols. For every $f\in\fun$, $n\in\mathbb{N}$, (possibly empty) tuples of terms $\vec{t_0},\vec{t_1}$, and $x,y\in \var$,

\begin{enumerate}[I.]\addtocounter{enumi}{13}
    \item \label{ax:modfun} $\big(\underline{2^{-\Delta(f;n)}}\dotsub \underline{d}(x,y)\big)\wedge \big(\underline{d}\big(f(\vec{t_0},x,\vec{t_1}),f(\vec{t_0},y,\vec{t_1})\big)\dotsub \underline{2^{-n}}\big).$
\end{enumerate}

\noindent Lastly, the following defines interaction between $\d$ and other predicate symbols. For every $P\in\pred$, $n\in\mathbb{N}$, (possibly empty) tuples of terms $\vec{t_0},\vec{t_1}$, and $x,y\in \var$,

\begin{enumerate}[I.]\addtocounter{enumi}{14}
    \item \label{ax:modpred} $\big(\underline{2^{-\Delta(P;n)}}\dotsub \underline{d}(x,y)\big)\wedge \big(\big(P(\vec{t_0},x,\vec{t_1})\dotsub P(\vec{t_0},y,\vec{t_1})\big)\dotsub \underline{2^{-n}}\big).$
\end{enumerate}

The \defn{rules of inference} of continuous logic are as follows, where $\varphi$ and $\psi$ are wffs and $x$ is a variable symbol.
\begin{itemize}
    \item \emph{Modus ponens}
    \[\frac{\varphi, \ \psi\dotsub\varphi}{\psi}\]
    \item \emph{Generalization}
    \[\frac{\varphi}{\sup_x\varphi}\]
\end{itemize}

The set of \defn{provable} wffs, \defn{consistency}, \defn{inconsistency}, and \defn{consequences} are each defined as in the classical setting. Again, we direct the reader toward \cite{BenYaacov.Berenstein.Henson.Usvyatsov.2008} for explicit definitions.


\subsection{Metric structures}\label{sec:metstruct}

Nearly every structure in mathematical analysis extends either a psuedometric or metric space. Continuous logic was developed with the purpose of describing such structures. On this note, a signature must be able to speak about the continuity of maps on such structures.

\begin{definition}
    Let $(|\M|,d)$ and $(|\M'|,d')$ be pseudometric spaces of diameter $1$ and let $f:|\M|\to|\M'|$. A map $\Delta(f):\N\to\N$ is called a \defn{modulus of continuity} for $f$ if for every $a,b\in |\M|$, $d(a,b)<2^{-\Delta(f;n)}$ implies that $d'\big(f(a),f(b)\big)\leq 2^{-n}$.
\end{definition}

An \defn{interpretation} of $L$ is a map $\cdot^\M$ with domain $\pred\cup\fun\cup\con$ such that for some \defn{universe} $|\M|$, each of the following hold.
    \begin{itemize}
        \item For every predicate symbol $P$, $P^\M:|\M|^{\eta(P)}\to [0,1]$.
        \item For every function symbol $f$, $f^\M:|\M|^{\eta(f)}\to |\M|$.
        \item For every constant symbol $c$, $c^\M\in |\M|$.
    \end{itemize}
$\cdot^\M$ is a \defn{continuous interpretation} if, moreover, each of the following hold.
    \begin{itemize}
        \item $\d^\M:=d$ is a pseudometric.
        \item For every predicate symbol $P$, $\Delta(P)$ is a modulus of continuity for $P$.\footnote{Here the domain of $P^\M$ is considered as the pseudometric space $\big(|\M|^{\eta(P)},(\d^\M)^{\eta(P)}\big)$ and the range the metric space $\big([0,1],|\cdot|\big)$.}
        \item For every function symbol $f$, $\Delta(f)$ is a modulus of continuity for $f$.
    \end{itemize}
When $\cdot^\M$ is an interpretation, the quintuple
    \[\M=\big(|\M|,d,\big\{P^\M:P\in\pred\setminus\{\d\}\big\},\big\{f^\M:f\in\fun\big\},\big\{c^\M:c\in\con\big\}\big)\]
is an \defn{$L$-pre-structure}. Moreover, if $\cdot^\M$ is a continuous interpretation, $\M$ is a \defn{continuous} $L$-pre-structure. Lastly, if $\cdot^\M$ is a continuous interpretation and $\big(|\M|,d\big)$ is a complete metric space, then $\M$ is an \emph{$L$-structure}. If $|\M|$ is countable, $\M$ is called \defn{weak}.
    
When $\M$ is an $L$-pre-structure, $\pred^\M$ is the set of \defn{predicates} of $\M$, $\fun^\M$ the set of \defn{functions} of $\M$, and $\con^\M$ the set of \defn{distinguished points} of $\M$.

At times, the language of \emph{non}-continuous pre-structures is dropped, and every pre-structure is assumed to be continuous. Also what were given here as ``$L$-structures'' are often designated as ``metric $L$-structures''. In this manuscript, however, we will assume every structure is interpreting a continuous signature, so we drop the prefix ``metric''.

\begin{example}\label{ex:structures}
    The following are examples of structures (for some related signature).
    \begin{itemize}
        \item A complete metric space of diameter $1$ with no additional structure.
        \item The natural numbers with the discrete metric and addition and multiplication as binary functions.
        \item The unit ball of a Banach space over $\R$ or $\C$, the norm as the metric, as functions all binary maps of the form 
        \[f_{\alpha,\beta}(x,y)=\alpha x+\beta y\]
        where $|\alpha|+|\beta|\leq 1$ are scalars, and the additive identity $0$ is a distinguished point.
        \item The unit ball of a $C^*$-algebra with the standard norm as the metric, and multiplication and the $*$-map included as functions.
        \item When $(\Omega, \mathcal{B}, \mu)$ is a probability space, let $M$ be its measure algebra and $d$ the measure of symmetric difference. Then $(M,d)$ along with $\mu$ as a predicate, $\cap$, $\cup$, and $\cdot^c$ as functions, and $0$ and $1$ as distinguished points is a structure.
    \end{itemize}
\end{example}

When $\M$ is a structure, an \defn{assignment} (on $\M$) is a map $\sigma:\var\to|\M|$. When $\sigma$ is an assignment, $x\in\var$, and $a\in|\M|$, the assignment $\sigma(x\mapsto a)$ is defined as follows.
    \[\sigma(x\mapsto a; \ y):=\begin{cases}
    a & \text{ if }y=x,\\
    \sigma(y) & \text{ otherwise.}
    \end{cases}\]
    Given a term $t$, the \defn{interpretation} of $t$ in $\M$ with $\sigma$ ($t^{\M,\sigma}$) is defined recursively as follows.
    \begin{itemize}
        \item If $t\in \con$, then $t^{\M,\sigma}:=t^\M$.
        \item If $t\in \var$, then $t^{\M,\sigma}:=\sigma(t)$.
        \item If $t=f(t_0,...,t_n)$, then $t^{\M,\sigma}:=f^\M\big(t_0^{\M,\sigma},...,t_n^{\M,\sigma}\big)$.
    \end{itemize}

For every $L$-pre-structure $\M$, assignment $\sigma$, and wff $\varphi$, the \defn{value} (or \defn{truth value}) of $\varphi$ in $\M$ with $\sigma$ ($\varphi^{\M,\sigma}$) is defined recursively as follows.
    \begin{itemize}
        \item $\big(P(t_0,...,t_n)\big)^{\M,\sigma}:=P^{\M}\big(t_0^{\M,\sigma},...,t_n^{\M,\sigma}\big)$.
        \item $(\neg\varphi)^{\M,\sigma}:=1-\varphi^{\M,\sigma}$.
        \item $\big(\half\varphi\big)^{\M,\sigma}:=\frac{1}{2}\cdot\varphi^{\M,\sigma}$.
        \item $(\varphi\dotsub\psi)^{\M,\sigma}:=\max\big\{\varphi^{\M,\sigma}-\psi^{\M,\sigma},0\big\}$.
        \item $\big(\sup_x\varphi\big)^{\M,\sigma}:=\sup_{a\in|\M|}\varphi^{\M,\sigma(x\mapsto a)}$.
        \item $\big(\inf_x\varphi\big)^{\M,\sigma}:=\inf_{a\in|\M|}\varphi^{\M,\sigma(x\mapsto a)}$.
    \end{itemize}
    When $\varphi^{\M,\sigma}=0$, $\M$ with $\sigma$ \defn{satisfies} $\varphi$ ($\M,\sigma\vDash\varphi$).  Also if $\varphi$ is a wff with free variables $\vec{x}$ and $\M$ an $L$-structure, $\varphi^{\M}(\vec{a})$ means $\varphi^{\M,\sigma(\vec{x}\mapsto\vec{a})}$ for any assignment $\sigma$.

\begin{remark}
    Instead of beginning with a signature and discussing the structures over that signature, one may begin with a structure and define the signature \emph{of} that structure. This is done as in the classical case, though normally the set of constants is assumed to be at most countable. Thus the interpretations of the constants are \emph{not} the entire universe.
\end{remark}

The following interpretations of shorthand sentences can be easily verified by direct computation.
    \begin{enumerate}[(a)]
        \item $(\varphi\vee\psi)^\M=\max\big\{\varphi^\M,\psi^\M\big\}$.
        \item $(\varphi\wedge\psi)^\M=\min\big\{\varphi^\M,\psi^\M\big\}$.
        \item $(\varphi\bic \psi)^\M=\big|\varphi^\M - \psi^\M\big|$.
        \item \label{subex:zero} $\underline{0}^\M=0$.
        \item \label{subex:one} $\underline{1}^\M=1$.
        \item $\varphi\dotplus\psi=\min\big\{\varphi^\M+\psi^\M,1\big\}$.
        \item $m\varphi=\min\big\{m\cdot \varphi^\M,1\big\}$.
        \item \label{subex:dyad} $\big(\underline{\frac{\ell}{2^k}}\big)^\M=\frac{\ell}{2^k}$.
    \end{enumerate}
Note that all of the above also hold for wffs when interpreted along with an assignment.

It is also important to note when one structure embeds into another.

\begin{definition}
    Let $\M$ and $\Nstruct$ be continuous $L$-pre-structures. $\iota:|\M|\to|\Nstruct|$ is an \defn{$L$-morphism} if each of the following hold.
    \begin{itemize}
        \item For every $f\in\fun$ and $a_0,...,a_{\eta(f)-1}\in |\M|$,
        \[\iota \big(f^\M(a_0,...,a_{\eta(f)-1})\big)=f^\Nstruct\big(\iota (a_0),...,\iota (a_{\eta{f}-1})\big).\]
        \item For every $P\in\pred$ and $a_0,...,a_{\eta(P)-1}\in |\M|$,
        \[\iota \big(P^\M(a_0,...,a_{\eta(P)-1})\big)=P^\Nstruct\big(\iota (a_0),...,\iota (a_{\eta{P}-1})\big).\]
    \end{itemize}
    $\iota $ is an \defn{elementary} $L$-morphism if, moreover, for every $\M$-assignment $\sigma$, $\varphi^{\M,\sigma}=\varphi^{\Nstruct,\iota \circ\sigma}$.
\end{definition}

\subsection{Computable analysis and presentations}

Computable analysis is summarized well in \cite{Weihrauch.2000}. For our purposes, however, we need only mention the definition of a computable real number.

\begin{definition}
    A real number $r$ is \defn{computable} if there is an effective procedure which, given $k\in\N$, outputs a rational $q\in\Q$ such that
    \[|r-q|<2^{-k}.\]
    When $A$ is a countable set, a map $f:A\to\R$ is then \defn{computable} if there is an effective procedure which, given $a\in A$ and $k\in\N$, outputs a rational $q\in \Q$ such that
    \[|f(r)-q|<2^{-k}.\]
\end{definition}

Since metric structures often have uncountable domains, so a method for discussing effectivity on such structures was introduced by \cite{Melnikov.2013}, recently also seen in \cite{Franklin.McNicholl.2020}. An effectively numbered signature is necessary for this method.

\begin{definition}\label{def:effnum}
    A signature $L$ is \defn{effectively numbered} if there is an effective mapping of the natural numbers onto $\pred\cup\fun\cup\con$ and, moreover, an effective procedure which, given the code of a predicate or function symbol, outputs that symbol's arity and an index of a Turing machine which serves as a modulus of continuity for that symbol.
\end{definition}

We now introduce computable presentations. From here we will assume we are working under a fixed effectively numbered signature $L$.

\begin{definition}
    Given an $L$-structure $\M$ and $A\subseteq |\M|$, the \defn{algebra generated by} $A$ is the smallest subset of $|\M|$ containing $A$ that is closed under every function of $\M$.
    
    A pair $(\M,g)$ is called a \defn{presentation} of $\M$ if $g:\mathbb{N}\to |\M|$ is a map such that the algebra generated by $\ran(g)$ is dense. A presentation of $\M$ is denoted $\M^\sharp$. Every point in $\ran(g)$ is called a \defn{distinguished point} of the presentation, and each point in the algebra generated by the distinguished points is called a  \defn{rational point} of the presentation ($\Q(\M^\sharp)$).
\end{definition}

Notably, $\ran(g)$ need not be dense, but the \emph{algebra} it generates does.

\begin{definition}
    A presentation $\M^\sharp$ is \defn{computable} if the predicates of $\M$ are uniformly computable on the rational points of $\M^\sharp$.\footnote{Since the metric is a binary predicate on $\M$, this entails that the distance between any two rational points is uniformly computable.}
    
    An \defn{index} of a computable presentation $\M^\sharp$ is an index of a Turing machine which, given a code of $P\in\pred$, codes of $a_0,...,a_{\eta(P)-1}\in \Q(\M^\sharp)$, and $k\in\N$, outputs a code of a rational $q$ such that
    \[\big|P^\M(a_0,...,a_{\eta(P)-1})-q\big|<2^{-k}.\]
\end{definition}

\begin{example}
    Let $\X$ be the metric structure consisting of the unit ball of a finite-dimensional Banach space over $\R$, its norm as the metric, as functions all binary maps of the form
    \[f_{p,q}(x,y)=px+qy\]
    where $p,q\in \Q$ and $|p|+|q|\leq 1$, and $0$ as the only distinguished point. Let $N$ be the dimension of $|\X|$. Then define $g:\N\to |\X|$ as $g(n)=e_n$ for every $n\leq N-1$ and $g(n)=e_{N-1}$ for every $n\geq N$. Clearly the algebra generated by $\ran(g)$ is dense in $|\X|$. Moreover, with a bit of careful calculation, one may see that this presentation is computable.
\end{example}

\section{Previous completeness results}\label{sec:completeres}

We now recall many results related to completeness which were proven in \cite{BenYaacov.Pedersen.2010}. We have altered some of the notation in order to make these results more applicable to our work, but the results proven remain the same. Important to this work is the introduction of the formal notion of dyadic numerals.

\begin{definition}
    The \defn{dyadic numerals} ($\dyad$) are all sentences of the form $\underline{\frac{\ell}{2^k}}$ for $\ell,k\in\N$. When $\p\in\dyad$, by $p$ we mean the real number such that for every $L$-pre-structure $\M$, $\p^\M=p$.
\end{definition}

Maximal consistency is defined similarly to the classical case, but with an extra condition concerning limiting behavior.

\begin{definition}\label{def:maxcon}
    A set of wffs $\Gamma$ is \defn{maximally consistent} if for every pair of wffs $\varphi$ and $\psi$, the following hold.
    \begin{enumerate}[(i)]
        \item \label{lem:maxconsp1} If $\Gamma\vdash \varphi\dotsub \underline{2^{-k}}$ for every $k\in\N$, then $\varphi\in \Gamma$.
        \item \label{lem:maxconsp2} Either $\varphi\dotsub \psi\in \Gamma$ or $\psi\dotsub\varphi\in \Gamma$.
    \end{enumerate}
\end{definition}

Notably, without condition (\ref{lem:maxconsp1}), we would not gain the intuitive property that if $\Gamma$ is maximally consistent, then for every $\varphi\notin \Gamma$, $\Gamma\cup\{\varphi\}$ is inconsistent. Ben Yaacov and Pedersen implement a continuous version of a Henkin construction to prove their completeness theorem. To accomplish this, Henkin witnesses must be added to the signature.

\begin{definition}
    Given a signature $L$, the \emph{Henkin extended signature of $L$} ($L^+$) is the smallest signature that extends $L$ and that, for every combination of $L^+$-wff $\varphi$, variable symbol $x$, and $\p,\q\in \dyad$, contains a unique constant symbol $c_{\varphi,x,\p,\q}$.
    
    When $\Gamma$ is a set of $L^+$-wffs, we say it is \emph{Henkin complete} if for every $L^+$-wff $\varphi$, every variable symbol $x$, and every $\p,\q\in \dyad$,
    \[\big(\sup_x\varphi\dotsub \q\big)\wedge\big(\p\dotsub\varphi[c_{\varphi,x,\p,\q}/x]\big)\in \Gamma.\]
\end{definition}

We now note a relevant lemma and theorem from Ben Yaacov and Pedersen.

\begin{lemma}[(ii) of Lemma 8.5, \cite{BenYaacov.Pedersen.2010}]\label{lem:thisorthat}
    Let $T$ be an $L$-theory. Then for every pair of $L$-wffs $\varphi$ and $\psi$, either $T\cup\{\varphi\dotsub \psi\}$ or $T\cup\{\psi\dotsub\varphi\}$ is consistent.
\end{lemma}

\begin{theorem}[From Theorem 8.10 and Proposition 9.2, \cite{BenYaacov.Pedersen.2010}]\label{thm:gammaext}
    Let $T$ be an $L$-theory. Then there exists a maximally consistent, Henkin complete set of $L^+$-wffs $\Gamma$ which extends $T$.
\end{theorem}

In what follows, the original Henkin model created will be a continuous $L^+$-pre-structure. To make the move to a genuine $L^+$-structure, the following theorem is needed.

\begin{theorem}[Theorem 6.9, \cite{BenYaacov.Pedersen.2010}]\label{thm:metcomp}
   Let $\M'$ be a continuous $L$-pre-structure. Then there is an $L$-structure $\M$ and an elementary $L$-morphism of $\M'$ into $\M$.
\end{theorem}

We now summarize the construction of the Henkin model in \cite{BenYaacov.Pedersen.2010}. Completeness follows.

\begin{definition}
    Let $\Gamma$ be a maximally consistent, Henkin complete set of $L^+$-wffs. Define the \emph{Henkin continuous $L^+$-pre-structure over $\Gamma$} ($\M_\Gamma'$) as follows.
    \begin{itemize}
        \item $|\M_\Gamma'|$ is the set of all terms of $L^+$.
        \item For every constant symbol $c$ of $L^+$, $c^{\M_\Gamma'}:=c$.
        \item For every function symbol $f$ of $L^+$, define $f^{\M_\Gamma'}$ for each $t_0,...,t_{\eta(f)-1}\in |\M_\Gamma'|$ as
        \[f^{\M_\Gamma'}\big(t_0,...,t_{\eta(f)-1}\big):=f(t_0,...,t_{\eta(f)-1}).\]
        \item For every predicate symbol $P$ of $L^+$, define $P^{\M_\Gamma'}$ for each $t_0,...,t_{\eta(P)-1}\in |\M_\Gamma'|$ as \[P^{\M_\Gamma'}\big(t_0,...,t_{\eta(P)-1}\big):=\sup\big\{ p\in [0,1] : \p\in\dyad \text{ and } \p\dotsub P(t_0,...,t_{\eta(P)-1})\in \Gamma \big\}.\]
    \end{itemize}
    
The \emph{basic assignment} on $\M_\Gamma'$ is defined as $\sigma(x):=x$ for every variable symbol $x$ of $L^+$. By a slight abuse of notation, when $\M_\Gamma'$ is a Henkin continuous $L^+$-pre-structure, by $\varphi^{\M_\Gamma'}$ we mean $\varphi^{\M_\Gamma',\sigma}$, and by $\M_\Gamma'\vDash \varphi$ we mean $\M_\Gamma',\sigma\vDash \varphi$, where $\sigma$ is the basic assignment.

The \emph{Henkin $L^+$-structure over $\Gamma$} ($\M_\Gamma$) is the structure induced by the metric completion of $\big(|\M_\Gamma'|,\d^{\M_\Gamma'}\big)$ and the elementary morphism given in Theorem \ref{thm:metcomp}.
\end{definition}

\begin{theorem}[Theorem 9.4, \cite{BenYaacov.Pedersen.2010}]\label{prop:henksat}
    Let $\Gamma$ be a maximally consistent, Henkin complete set of $L^+$-wffs. Then $\M_\Gamma \vDash \Gamma$.
\end{theorem}

\begin{corollary}[Completeness of Continuous Logic, Theorem 9.5, \cite{BenYaacov.Pedersen.2010}]\label{thm:complete}
    A set of $L$-wffs is consistent if and only if it is (completely) satisfiable.
\end{corollary}

Ben Yaacov and Pedersen then introduce important maps from sets of $L$-wffs into $[0,1]$. These maps serve as upper bounds on relative provability and interpretation of sentences following from those sets of $L$-wffs.

\begin{definition}
    Let $\Gamma$ be a set of $L$-wffs. The \emph{degree of truth with respect to $\Gamma$} ($ \ \cdot\text{ }^\circ_\Gamma$) is a map from wffs to $[0,1]$, defined as
    \[\varphi_\Gamma^\circ:=\sup\big\{\varphi^{\M,\sigma}:\mathfrak{M},\sigma\vDash \Gamma \big\}.\]
    	
    The \emph{degree of provability with respect to $\Gamma$} ($ \ \cdot\text{ }^\circledcirc_\Gamma$) is a similar map, defined as
	\[\varphi_\Gamma^\circledcirc:=\inf\big\{p\in[0,1]:\p\in\dyad \text{ and } \Gamma \vdash\varphi\dotsub \p\big\}.\]
\end{definition}

The Completeness Theorem then implies that these maps are the same.

\begin{corollary}[Corollary 9.8, \cite{BenYaacov.Pedersen.2010}]\label{prop:circequiv}
    For any $L^+$-wff $\varphi$ and set of $L$-wffs $\Gamma$, $\varphi_\Gamma ^\circ=\varphi_\Gamma ^\circledcirc$.
\end{corollary}


\begin{definition}
    A set of $L$-wffs $\Gamma$ is \defn{complete} if there is a structure $\M$ and assignment $\sigma$ such that for every $L$-wff $\varphi$,
    \[\varphi^\circ_T=\varphi^{\M,\sigma}.\]
    $\Gamma$ is \defn{incomplete} if it is not complete.
\end{definition}

In contrast to the classical case, even if a theory is complete, its set of consequences may not be maximally consistent. This is due to the limiting behavior condition discussed in Definition \ref{def:maxcon}. The Deduction Theorem for continuous logic encounters a similar issue.

\begin{theorem}[Deduction Theorem, Theorem 8.1, \cite{BenYaacov.Pedersen.2010}]
   Let $\Gamma$ be a set of $L$-wffs. Then for every $L$-wff $\psi$, $\Gamma\cup\{\psi\}\vdash \varphi$ if and only if $\Gamma\vdash \varphi\dotsub m\psi$, for some $m\in\mathbb{N}$.
\end{theorem}

We also note the Generalization Theorem, which will be useful in future work.

\begin{lemma}[Generalization Theorem, Lemma 8.2, \cite{BenYaacov.Pedersen.2010}]
    Let $\Gamma$ be a set of $L^+$-wffs and $\varphi$ an $L^+$-wff. If $x$ does not appear freely in $\Gamma$ and $\Gamma\vdash \varphi$, then $\Gamma\vdash \sup_x \varphi$.
\end{lemma}

And lastly, we note the following lemma of Calvert's.

\begin{lemma}[Lemma 4.6, \cite{Calvert.2011}]\label{lem:sigext}
    There is an effective procedure which extends $L$ to its Henkin extended signature $L^+$. 
\end{lemma}

\section{Main result}\label{sec:main}

\subsection{Model-theoretic preliminaries}\label{sub:mod}

There are four important model-theoretic propositions which extend the results of \cite{BenYaacov.Pedersen.2010} and are useful for the construction of the effective completeness theorem.

\begin{proposition}\label{prop:circgeq}
    Let $\Gamma$ be a set of $L$-wffs and $B$ a finite set of $L$-wffs. Then for every $L$-wff $\varphi$, 
    \[\varphi_{\Gamma \cup B}^\circ\leq \Big(\varphi\dotsub \bigvee_{\theta\in B}\theta\Big)_\Gamma^\circ.\]
\end{proposition}

\begin{proof}
    
    Fix an $L$-wff $\varphi$. If $\varphi^\circ_{\Gamma\cup B}=0$, the result follows trivially. Thus suppose $\varphi^\circ_{\Gamma\cup B}>0$. Notably, this implies that $\Gamma\cup B$ is consistent. Fix $p\in\mathrm{Dyad}_L$ such that $\overline{p}<\varphi_{\Gamma \cup B}^\circ$. Then there is some $L$-structure $\M$ and assignment $\sigma$ such that $\M,\sigma\vDash \Gamma\cup B$ while $\varphi^{\M,\sigma}>\overline{p}$. But, clearly, since  $\M,\sigma\vDash B$, $\Big(\bigvee_{\theta\in B}\theta\Big)^{\M,\sigma}=0$. Hence, $\Big(\varphi\dotsub \bigvee_{\theta\in B}\theta\Big)^{\M,\sigma}>\overline{p}$. Then since $\M,\sigma\vDash \Gamma$, this implies that $\Big(\varphi\dotsub \bigvee_{\theta\in B}\theta\Big)^\circ_{\Gamma}>\overline{p}$. Since this is true for every $\overline{p}<\varphi_{\Gamma \cup B}^\circ$, we have that $\varphi_{\Gamma \cup B}^\circ\leq \Big(\varphi\dotsub \bigvee_{\theta\in B}\theta\Big)_\Gamma^\circ$.

\end{proof}


\begin{proposition}\label{prop:eqone}
    Let $\Gamma$ be a set of $L$-wffs and $B$ a finite set of $L$-wffs such that $\Gamma\cup B$ is consistent. Then there are infinitely many $L$-wffs $\varphi$ such that 
    \[\Big(\varphi\dotsub \bigvee_{\theta\in B}\theta\Big)_\Gamma^\circ=1.\]
\end{proposition}

\begin{proof}
    Recall that $\Gamma\cup B$ is consistent if and only if there is some $L$-wff $\varphi$ such that $\Gamma\cup B \nvdash \varphi$. By Corollary \ref{prop:circequiv}, $\varphi_{\Gamma\cup B}^\circ>\frac{1}{M}$, for some $M\in\mathbb{N}$. It follows that for every $m\geq M$, $(m\varphi)_{\Gamma\cup B}^\circ =1$. Hence, by Proposition \ref{prop:circgeq}, $\Big(\varphi\dotsub \bigvee_{\theta\in B}\theta\Big)_\Gamma^\circ=1$, for every $m\geq M$.
\end{proof}

\begin{proposition}\label{prop:supequiv}
    Let $L$ be a signature, $T$ an $L$-theory, and $\varphi$ an $L$-wff with free variables $\vec{x}$. Then
    \[\varphi^\circ_T=\big(\sup_{\vec{x}}\varphi\big)^\circ_T.\]
\end{proposition}

\begin{proof}
    Fix a signature $L$, an $L$-theory $T$, and an $L$-wff $\varphi$ with free variables $\vec{x}$. Notice that since $T$ contains only $L$-sentences, none of $\vec{x}$ appear freely in $T$. It follows via Corollary \ref{prop:circequiv} and the Generalization Theorem that
    \eqn{\varphi^\circ_T&=\inf\big\{\overline{p}:p\in\mathrm{Dyad}_L \text{ and } T\vdash\varphi\dotsub p\big\}\\
    &=\inf\big\{\overline{p}:p\in\mathrm{Dyad}_L \text{ and } T\vdash\sup_{\vec{x}}\varphi\dotsub p\big\}\\
    &=\big(\sup_{\vec{x}}\varphi\big)^\circ_T.}
\end{proof}

\begin{proposition}\label{lem:repeq}
    Let $T$ be an $L$-theory. Then for every $L^+$-wff $\theta$,
    \[\big(\sup_{\vec{x}}\theta[\vec{x}/\vec{c}]\big)^\circ_T= \theta^\circ_T,\]
    where $\vec{c}$ is the tuple of all constants from $L^+$, but not in $L$, appearing in $\theta$.
\end{proposition}

\begin{proof}
    Fix an $L$-theory $T$ and an $L^+$-wff $\theta$. Recall that no variable from $\vec{x}$ appears freely in $T$, nor does any constant in $\vec{c}$ appear in $T$, since it is an $L$-theory. Hence for every $L$-structure $\M$ such that $\M\vDash T$, and every tuple $\vec{a}\in |\M|$, there is an $L^+$-structure $\M^+_{\vec{a}}$ such that $\vec{c}^{\M^+_{\vec{a}}}=\vec{a}$ and $\M^+_{\vec{a}}\upharpoonright_L =\M$. Then
    \eqn{\big(\sup_{\vec{x}}\theta[\vec{x}/\vec{c}]\big)^\circ_T&=\sup\Big\{\big(\sup_{\vec{x}}\theta[\vec{x}/\vec{c}]\big)^{\M}:\mathfrak{M}\vDash T \Big\}\\
    &=\sup\Big\{\sup_{\{\sigma(\vec{x}\mapsto\vec{a}):\vec{a}\in|\M|\}}\big(\theta[\vec{x}/\vec{c}]\big)^{\M,\sigma(\vec{x}\mapsto\vec{a})}:\mathfrak{M}\vDash T\Big\}\\
    &=\sup\Big\{\big(\theta[\vec{x}/\vec{c}]\big)^{\M,\sigma(\vec{x}\mapsto\vec{a})}:\mathfrak{M},\sigma(\vec{x}\mapsto\vec{a})\vDash T, \ \vec{a}\in |\M|\Big\}\\
    &\leq \sup\Big\{\big(\theta[\vec{x}/\vec{c}]\big)^{\M^+_{\vec{a}},\sigma(\vec{x}\mapsto\vec{a})}:\M^+_{\vec{a}},\sigma(\vec{x}\mapsto\vec{a})\vDash T, \ \vec{a}\in |\M| \Big\}\\
    &=\sup\Big\{\theta^{\M^+_{\vec{a}},\sigma}:\M^+_{\vec{a}},\sigma\vDash T \Big\}\\
    &\leq \sup\Big\{\theta^{\M^+,\sigma}:\M^+,\sigma\vDash T \Big\}\\
    &=\theta^\circ_T.}
    Now notice that for any $L^+$-structure $\M^+$ and assignment $\sigma$, there is an assignment $\sigma(\vec{x}\mapsto \vec{c}^{\M^+_{\vec{a}}})$. Then the $L$-structure $\M^+\upharpoonright_L$ is such that $\theta^{\M^+,\sigma}=\big(\theta[\vec{x}/\vec{c}]\big)^{\M^+\upharpoonright_L,\sigma(\vec{x}\mapsto \vec{c}^{\M^+_{\vec{a}}})}$. It follows that
    \eqn{\theta^\circ_T&=\sup\Big\{\theta^{\M^+,\sigma}:\M^+,\sigma\vDash T \Big\}\\
    &= \sup\Big\{\big(\theta[\vec{x}/\vec{c}]\big)^{\M^+\upharpoonright_L,\sigma(\vec{x}\mapsto \vec{c}^{\M^+_{\vec{a}}})}:\M^+\upharpoonright_L,\sigma(\vec{x}\mapsto \vec{c}^{\M^+_{\vec{a}}})\vDash T \Big\}\\
    &\leq \sup\Big\{\big(\theta[\vec{x}/\vec{c}]\big)^{\M,\sigma}:\mathfrak{M},\sigma\vDash T\Big\}\\
    &=\big(\theta[\vec{x}/\vec{c}]\big)^\circ_T.}
    But by Proposition \ref{prop:supequiv}, $\big(\theta[\vec{x}/\vec{c}]\big)^\circ_T=\big(\sup_{\vec{x}}\theta[\vec{x}/\vec{c}]\big)^\circ_T$. Therefore, $\theta^\circ_T\leq \big(\sup_{\vec{x}}\theta[\vec{x}/\vec{c}]\big)^\circ_T$. The claim follows.
\end{proof}

\subsection{Effectively extending theories}\label{sub:dec}

Since any $L$-theory $T$ has an associated degree of truth map $\cdot^\circ_T$, to analyze the effectiveness of a theory we will actually consider the effectiveness of the related degree of truth map. The following definition is given in \cite{BenYaacov.Pedersen.2010}.

\begin{definition}
    An $L$-theory $T$ is \defn{decidable} if $\cdot^\circ_T$ is a computable map from the set of wffs to $[0,1]$.
\end{definition}

Since there are uncountably-many such maps, we introduce a naming system.

\begin{definition}
Given an $L$-theory $T$, we say that $X\in \mathbb{N}^\mathbb{N}$ is a \emph{name} of $T$ if the following hold.
\begin{itemize}
    \item For every $n,k\in\mathbb{N}$, there is some $m\in\mathbb{N}$ such that $\langle n,k,m \rangle\in \mathrm{ran}(X)$.
    \item For every $n,k,m\in\mathbb{N}$, if $\langle  n,k,m \rangle\in\mathrm{ran}(X)$, then $ q_m\in\big[ \ (\varphi_n)^\circ_T-2^{-k} \ , \ (\varphi_n)^\circ_T+2^{-k} \ \big]$.
\end{itemize}
\end{definition}

\begin{proposition}
    An $L$-theory is decidable if and only if it has a computable name.
\end{proposition}

\begin{proof}
    For the forward direction, suppose $\cdot^\circ_T$ is computable. Fix a witness of this computability. For every code of a pair, define $X(\langle n,k \rangle):=\langle n,k,m \rangle$ where the witness outputs $q_m$, given $\varphi_n$ and precision parameter $k$. On every natural number which doesn't code a pair, let $X$ be $0$. For the reverse direction, suppose $X$ is a computable name. Given $\varphi_n$ and a precision parameter $k$, begin computing $\mathrm{ran}(X)$ until a code of a triple of the form $\langle n,k,m \rangle$ is output. It follows that $q_m\in\big[ \ (\varphi_n)^\circ_T-2^{-k} \ , \ (\varphi_n)^\circ_T+2^{-k} \ \big]$.
\end{proof}

Let $L^+$ be the Henkin extended signature effectively given by Lemma \ref{lem:sigext}, and let $(\theta_n)_{n\in\mathbb{N}}$ be an effective enumeration of the $L^+$-wffs. The next lemma we present is similar to Lemma 4.7 in \cite{Calvert.2011}. However, in our case, the construction is with respect to any \emph{name} of an $L$-theory, $X\in \mathbb{N}^\mathbb{N}$. Moreover, careful consideration is taken with respect to when two $L^+$-sentences are provably equivalent with respect to a given $L^+$-theory.

The basic idea of the proof is the following. Given the degree of truth of a theory $\cdot^\circ_T$, find the first $L^+$-wff of the form $\varphi\dotsub \psi$ which has a \emph{strictly positive} degree of truth. It follows that there is some structure $\M$ and assignment $\sigma$ such that $\M,\sigma \vDash T$ while $\M,\sigma \nvDash \varphi\dotsub \psi$. Hence $\M,\sigma \vDash T\cup\{\psi\dotsub \varphi\}$. Moreover, $\cdot^\circ_{T\cup\{\psi\dotsub \varphi\}}$ is shown to be effective in $\cdot^\circ_T$. Thus we may effectively complete $T$ as an $L^+$-theory.

\begin{lemma}\label{lem:maxconsset}
    There is an effective procedure which given $X$, a name of an $L$-theory $T$, outputs $\Phi(X)\subseteq \N$ such that $T\cup\{\theta_n:n\in \Phi(X)\}$ is consistent, and for every pair of $L^+$-wffs $\varphi$ and $\psi$, either $\varphi$ and $\psi$ are provably equivalent with respect to $T\cup\{\theta_n:n\in \Phi(X)\}$, or exactly one of $\varphi\dotsub \psi$ or $\psi\dotsub\varphi$ is in $\{\theta_n:n\in \Phi(X)\}$.
\end{lemma}

\begin{proof}
    We proceed via partial effective recursion. First define $\Phi_{0}(X):=\emptyset$, for every $X\in \mathbb{N}^\mathbb{N}$. As the recursive assumption, we suppose that at stage $s$, if $X$ is a name of an $L$-theory $T$, then $\Phi_s(X)$ is defined, finite, and $T\cup\{\theta_n:n\in \Phi_{s}(X)\}$ is consistent. At stage $s+1$, the following procedure attempts to construct $\Phi_{s+1}(X)$.
    
    For every pair of $L^+$-wffs $\varphi$ and $\psi$, define the real number \[r_{\varphi,\psi,f,s+1}:=\Big(\sup_{\vec{x},\vec{y}} \ \sup_{\vec{z}}\Big(\Big((\psi\dotsub\varphi)\dotsub \Big(\bigvee_{n\in \Phi_{s}(X)} \theta_n\Big)\Big)[\vec{z}/\vec{c}]\Big)\Big)_{T}^\circ,\]
    where $\vec{x}$ and $\vec{y}$ are the free variables appearing in $\psi\dotsub\varphi$ and $\bigvee_{n\in \Phi_{s}(X)} \theta_n$, respectively, $\vec{c}$ is the (possibly empty) tuple of constants from $L^+$ and not in $L$ appearing in $(\psi\dotsub\varphi)\dotsub \big(\bigvee_{n\in \Phi_{s}(X)} \theta_n\big)$, and $\vec{z}$ is a $|\vec{c}|$-tuple of variable symbols distinct from $\vec{x}$ and $\vec{y}$. Notably, these real numbers are computable in $f$, uniformly in $\varphi$, $\psi$, and $s$. To see this, notice that each recursively defined $\Phi_{s}(X)$ is finite, each free variable becomes bound by the quantifier, and every constant from $L^+$ not in $L$ is replaced by a variable and bound. Hence each formula checked above is actually an $L$-sentence, so $f$ can compute a rational approximation of $r_{\varphi,\psi,X,s+1}$ within $2^{-(s+2)}$. Call such a rational $q_{\varphi,\psi,X,s+1}$. Then search the pairs of $L^+$-wffs for the first pair $\varphi$ and $\psi$ such that $\varphi\dotsub \psi\notin \{\theta_n:n\in \Phi_{s}(X)\}$ and $q_{\varphi,\psi,X,s+1}\geq 2^{-(s+1)}$. By Proposition \ref{prop:eqone}, there are infinitely many $L^+$-wffs $\psi$ such that $r_{\underline{0},\psi,X,s+1}=1$, and hence such that $q_{\underline{0},\psi,X,s+1}\geq 2^{-(s+1)}$. Thus, when $X$ is a name of an $L$-theory, the procedure will halt. When such a pair $\varphi$ and $\psi$ is found, search the effective enumeration of the $L^+$-wffs for the index $m$ of $\varphi\dotsub \psi$ and define $\Phi_{s+1}(X):=\Phi_{s}(X)\cup \{m\}$. Clearly, if $\Phi_{s+1}(X)$ is defined, it is also finite, by construction. We now claim that this $\Phi$ witnesses the lemma. Fix a name of an $L$-theory $X\in \mathbb{N}^\mathbb{N}$.
    
    To see that ${T}\cup \{\theta_n:n\in \Phi(X)\}$ is consistent, we show that each ${T}\cup \{\theta_n:n\in \Phi_{s+1}(X)\}$ is consistent, for every $s\in\mathbb{N}$. We proceed inductively.
    
    Suppose ${T}\cup \{\theta_n:n\in \Phi_{s}(X)\}$ is consistent and fix $m\in \Phi_{s+1}(X)\setminus \Phi_{s}(X)$. By construction this $m$ is the index for $\varphi\dotsub \psi$ where $q_{\varphi,\psi,X,s+1}\geq 2^{-(s+1)}$. It follows by the definition of $q_{\varphi,\psi,X,s+1}$ and Propositions \ref{prop:supequiv} and \ref{lem:repeq}, Corollary \ref{prop:circequiv}, logical equivalence, and the Deduction Theorem, that we have each of the following.
    \eqn{\Big((\psi\dotsub\varphi)\dotsub& \Big(\bigvee_{n\in \Phi_{s}(X)} \theta_n\Big)\Big)_{T}^\circ \geq 2^{-(s+2)}\\
    &\implies \ \ \Big(\underline{2^{-(s+2)}}\dotsub \Big((\psi\dotsub\varphi)\dotsub \Big(\bigvee_{n\in \Phi_{s}(X)} \theta_n\Big)\Big)\Big)_{T}^\circ =0\\
    &\implies \ \ {T}\vdash \Big(\underline{2^{-(s+2)}}\dotsub \Big((\psi\dotsub\varphi)\dotsub \Big(\bigvee_{n\in \Phi_{s}(X)} \theta_n\Big)\Big)\Big)\dotsub \underline{2^{-(s+3)}}\\
    &\implies \ \ {T}\cup \{\theta_n:n\in \Phi_{s}(X)\}\cup\{\psi\dotsub \varphi\}\vdash \underline{2^{-(s+3)}}.}
    
    
    Therefore, ${T}\cup \{\theta_n:n\in \Phi_{s}(X)\}\cup\{\psi\dotsub \varphi\}$ is inconsistent. It follows by Lemma \ref{lem:thisorthat} that $\varphi\dotsub \psi$ is consistent with ${T}\cup \{\theta_n:n\in \Phi_{s}(X)\}$.
    
    We therefore need only show that for every pair of $L^+$-wffs $\varphi$ and $\psi$, either $\varphi$ and $\psi$ are provably equivalent with respect to $T\cup\{\theta_n:n\in \Phi(X)\}$, or exactly one of $\varphi\dotsub \psi$ or $\psi\dotsub\varphi$ is in $\{\theta_n:n\in \Phi(X)\}$. 
    
    Note that a pair of $L^+$-wffs $\varphi$ and $\psi$ is provably equivalent with respect to $T\cup\{\theta_n:n\in \Phi(X)\}$ if and only if for every $s\in\mathbb{N}$, there is some $S\in\mathbb{N}$ such that
    \[T\cup\{\theta_n:n\in \Phi_{S}(X)\}\vdash (\varphi\dotsub \psi)\dotsub \underline{2^{-(s+2)}} \ \ \ \ \text{ and } \ \ \ \ T\cup\{\theta_n:n\in \Phi_{s}(X)\}\vdash (\psi\dotsub \varphi)\dotsub \underline{2^{-(s+2)}}.\]
    Now fix a pair of $L^+$-wffs $\varphi$ and $\psi$ that are not provably equivalent with respect to $T\cup\{\theta_n:n\in \Phi(X)\}$. Then there must be some $s\in\mathbb{N}$ such that for every $S\in\mathbb{N}$, either
    \[T\cup\{\theta_n:n\in \Phi_{S}(X)\}\nvdash (\varphi\dotsub \psi)\dotsub \underline{2^{-(s+2)}} \ \ \ \ \text{ or } \ \ \ \ T\cup\{\theta_n:n\in \Phi_{S}(X)\}\nvdash (\psi\dotsub \varphi)\dotsub \underline{2^{-(s+2)}}.\]
    At least one of these two cases must hold for infinitely many $S\in\mathbb{N}$. Without loss of generality, since the cases are symmetric, suppose it is the latter. It follows by Corollary \ref{prop:circequiv} that for every $S\in\mathbb{N}$, $(\psi\dotsub \varphi)^\circ_{T\cup\{\theta_n:n\in \Phi_{S}(X)\}}\geq 2^{-(s+2)}.$ Hence by Proposition \ref{prop:circgeq}, for every $S\in\mathbb{N}$,
    \[\Big((\psi\dotsub\varphi)\dotsub \Big(\bigvee_{n\in\Phi_{S}(X)} \theta_n\Big)\Big)_{T}^\circ \geq 2^{-(s+2)}.\]
    Then by Propositions \ref{prop:supequiv} and \ref{lem:repeq}, for every $S\in\mathbb{N}$, $r_{\varphi,\psi,X,S+1}\geq 2^{-(s+2)}$.
    Thus for some $S\geq s+2$, $\varphi$ and $\psi$ will have to be the first pair such that $\varphi\dotsub \psi\notin \{\theta_n:n\in \Phi_{S}(X)\}$ and $q_{\varphi,\psi,X,S+1}\geq 2^{-(s+3)}\geq 2^{-(S+1)}$. It follows that the procedure will place the index for $\varphi\dotsub \psi$ into $\Phi_{S+1}(X)$, so $\varphi\dotsub \psi\in \{\theta_n:n\in \Phi(X)\}$.
\end{proof}

It should be noted that if $T$ is not complete, a name of $T$ does not specify a \emph{unique} consistent extension of $T$. The above procedure constructs a complete extension, which itself has a unique maximally consistent extension, but the procedure is dependent on the enumeration of the $L^+$-wffs. When that enumeration changes, if $T$ is not a complete theory, the above extension of $T$ may also change.

\subsection{Generalized effective completeness}\label{sub:eff}

We now come to our main result.

\begin{theorem}[Generalized Effective Completeness]\label{thm:geneffcomp}
    There is an effective procedure which, given a name $X\in \N^\N$ of an $L$-theory $T$, produces a presentation of an $L^+$-structure $\M$ such that $\M\vDash T$.
\end{theorem}

\begin{proof}
    Compute $L^+$ as in Lemma \ref{lem:sigext}. Given a name of an $L$-theory $X\in \mathbb{N}^\mathbb{N}$, let $\Phi(X)$ be as in Lemma \ref{lem:maxconsset}. Then, by Theorem \ref{thm:gammaext}, extend $T\cup\{\theta_n:n\in \Phi(X)\}$ to a maximally consistent, Henkin complete $L^+$-theory $\Gamma$. By Proposition \ref{prop:henksat}, $\M_\Gamma\vDash T$.
    
    Since $L^+$ is effectively numbered, the set of constants of $L^+$ is also effectively numbered, which we may effectively join to an effective numbering of the variable symbols. Let $g'$ be such an effective numbering. Then, for every $n\in\mathbb{N}$, define $g(n):=[g'(n)]$, the equivalence class of $g'(n)$ in $|\M_\Gamma|$. By construction, the algebra generated by $\mathrm{ran}(g)$ in $\M_\Gamma$ is the set of all equivalence classes of terms of $L^+$, that is, equivalence classes of the elements of $|\M_\Gamma'|$. It follows that this algebra is dense in $|\M_\Gamma|$, since by construction $|\M_\Gamma|$ is the metric completion of $|\M_\Gamma'|$. Thus $\big(\M_\Gamma,g\big)$ is a presentation of $\M_\Gamma$. We further claim that $\big(\M_\Gamma,g\big)$ is a computable presentation.
    
    Fix a code of an $N$-ary predicate symbol $P$, codes of rational points $[t_0],...,[t_{N-1}]$, and a precision parameter $k\in\mathbb{N}$. From these, use $g'$ to decode $L^+$-terms $t_0,...,t_{N-1}$ corresponding to $[t_0],...,[t_{N-1}]$.
    Then execute the following.
    
    Compute the finite set $D=\big\{\p\in\dyad:\text{ the denominator of } p \text{ is less than }2^{k+2}\big\}$. By the construction of $\Phi(X)$,  with access to an oracle that computes $X$ we may compute the least $M\geq k+2$ such that for all but one $\p\in D$, exactly one of $\p\dotsub P(t_0,...,t_{N-1})$ or $P(t_0,...,t_{N-1})\dotsub \p$ is in $\{\theta_n:n\in \Phi_{M+1}(X)\}$.\footnote{It may be that for some $\p\in D$, for every $M\in\mathbb{N}$, $\big(P(t_0,...,t_{N-1})\dotsub \bigvee_{n\in \Phi_{M}(X)} \theta_n \big)^\circ_{T}$ and $p$ differ by less than $2^{-(M+2)}$. This can only occur if $P(t_0,...,t_{N-1})$ and $\p$ are provably equivalent with respect to $T\cup\{\theta_n:n\in \Phi(X)\}$, which can happen for at most one $\p\in D$, since $T\cup\{\theta_n:n\in \Phi(X)\}$ is consistent.} Then compute the finite set $E=\big\{\p\in D:\p\dotsub P(t_0,...,t_{N-1})\in \{\theta_n:n\in \Phi_{M+1}(X)\}\big\}$. Notice that by construction 
    \[\underline{\max_{\p\in E}p}\dotsub P(t_0,...,t_{N-1})\in \Gamma \ \ \ \ \text{ and } \ \ \ \ P(t_0,...,t_{N-1})\dotsub \big(\underline{\min_{\p\in D\setminus E}p}\dotplus \underline{2^{-(k+2)}}\big)\in \Gamma.\]
    Therefore,
    \[\M_\Gamma\vDash \underline{\max_{\p\in E}p}\dotsub P(t_0,...,t_{N-1}) \ \ \ \ \text{ and } \ \ \ \ \M_\Gamma\vDash P(t_0,...,t_{N-1})\dotsub \Big(\underline{\min_{\p\in D\setminus E}p}\dotplus \underline{2^{-(k+2)}}\Big).\]
    It follows that
    \[ \max_{\p\in E}p\ \leq \  \big(P(t_0,...,t_{N-1})\big)^{\M_\Gamma} \ \leq \ \Big(\min_{\p\in D\setminus E}p+ 2^{-(k+2)}\Big).\]
    This implies that
    \[\big(P(t_0,...,t_{N-1})\big)^{\M_\Gamma} \ \in \ \Big[ \ \Big(\min_{\p\in D\setminus E}p- 2^{-(k+1)}\Big) \ , \ \Big(\min_{\p\in D\setminus E}p+ 2^{-(k+1)}\Big) \ \Big].\]

    Note lastly that this procedure was uniform in $X\in \mathbb{N}^\mathbb{N}$.
\end{proof}

It follows that the presentation given by the above theorem is Turing reducible to the name input. Hence if the name given is computable (meaning the theory is decidable), the presentation produced is also computable. Standard effective completeness then comes as a corollary.

\begin{corollary}[Effective Completeness of Continuous Logic]\label{thm:eff}
    Every decidable theory is modeled by a computably presentable structure.
\end{corollary}

\subsection*{Acknowledgement}

I would like to thank Timothy H. McNicholl for his generous advice, comments, critique, and encouragement on the completion of this paper.

\include{thebibliography}

\bibliographystyle{plain}
\bibliography{mybib2}

\end{document}